\documentclass[11pt,reqno]{amsart}
\usepackage{amssymb}
\usepackage{amsmath}
\usepackage{mathrsfs}
\usepackage{amssymb, amsmath, amsfonts}

\usepackage{mathabx}

\usepackage{color}

\usepackage[hidelinks]{hyperref}
\allowdisplaybreaks
\usepackage[numbers,sort&compress]{natbib}

\makeatletter
\def\tank#1{\protected@xdef\@thanks{\@thanks
 \protect\footnotetext[0]{#1}}}
\def\bigfoot{

 \@footnotetext}
\makeatother

\newcommand{\ea}{\end{array}}

\allowdisplaybreaks
\numberwithin{equation}{section}
\newtheorem{theorem}{Theorem}[section]
\newtheorem{lemma}{Lemma}[section]
\newtheorem{proposition}{Proposition}[section]

\newtheorem{remark}{Remark}[section]

\newtheorem{definition}{Definition}[section]

\def\beq{\begin{equation}}
\def\nneq{\end{equation}}

\def\bthm{\begin{theorem}}
\def\nthm{\end{theorem}}

\def\blem{\begin{lemma}}
\def\nlem{\end{lemma}}
\def\bprf{\begin{proof}}
\def\nprf{\end{proof}}
\def\bprop{\begin{prop}}
\def\nprop{\end{prop}}
\def\brmk{\begin{rem}}
\def\nrmk{\end{rem}}

\def\bexa{\begin{exa}}
\def\nexa{\end{exa}}
\def\bcor{\begin{cor}}
\def\ncor{\end{cor}}

\def\e{\varepsilon}

\def\RR{\mathbb{R}}

\def\e{{\varepsilon}}

\title[]{Transportation cost inequalities for mean reflection SPDEs with white noise}

 \author[B. Zhang]{Beibei Zhang}
    \address[]{Beibei Zhang, Department of Mathematics and Statistics, Suzhou university of Technology, Changshu, Jiangsu, 215500,  China.}
    \email{zhangbb@whu.edu.cn}

\author[B. Qian]{Bin Qian}
    \address[]{Bin Qian, Department of Mathematics and Statistics, Suzhou university of Technology, Changshu, Jiangsu, 215500,  China.}
    \email{binqiancn@126.com}

\date{}

\begin{document}
\maketitle

 \noindent {\bf Abstract:}
We establish a quadratic transportation cost inequality under the uniform norm for solutions to mean reflected stochastic partial differential equations, a new type of equation in which the compensating reflection part depends not on the paths but on the law of the solution.

 \vskip0.3cm

\vskip0.3cm

 \section{Introduction}

 In this paper, we consider mean reflected stochastic partial differential equations (SPDEs for short) driven by  space-time white noise, which are reflected SPDEs with a constraint on the law of the process $u$ rather than on its paths. More precisely, we consider the following SPDE:
\begin{equation}\label{SPDE}
\left\{\begin{split}
   &\frac{\partial u}{\partial t}(t, x)- \frac{\partial^2 u}{ \partial x^2} (t,x) + f(t,x,u(t,x))=\sigma(t,x,u(t,x)) \dot W(t, x)\\&\qquad\qquad\qquad \qquad\qquad\qquad\qquad\qquad\quad+\dot{K}(t,x),\quad x\in [0,1],t\geq0\\
   &u(0,x)=u_0(x),\quad x\in [0,1],\\
    &u(t,0)=u(t,1)=0,\quad t\ge 0,
 \end{split}\right.
\end{equation}
with \begin{align}\label{Restraincondition}
\mathbb{E}\left[h(t,x,u(t,x))\right]\geq 0, \quad \forall x\in[0,1], t\geq0.
\end{align}
 Here $\dot W(t,x)=\frac{\partial^2W}{\partial t\partial x} (t,x)$   denotes space-time white noise defined on a complete probability space
 $(\Omega, \mathcal F, \{\mathcal F_t\}_{t\ge0}, \mathbb P)$ with  $\mathcal F_t=\sigma\{W(s, x): 0\le s\le t, x\ge 0\}$,  i.e.,  $\dot W(t,x)$ is a random set function  on sets $A\subseteq  \mathbb{R}_+\times[0,1]$  (we always use $\mathbb{R}_+$ to denote $[0,\infty)$ throughout this paper) of finite Lebesgue measure $m$ such that
  (i) $\dot W(A)$ is an $N(0,m(A))$-random variable;
  (ii) if $A\cap B=\emptyset$, then $\dot W(A)$ and $\dot W(B)$ are independent and $\dot W(A\cup B)= \dot W(A) +\dot W(B)$.
  The initial condition $u_0(x)$ is a non-negative continuous function, which vanishes at points 0 and 1. The coefficients $f$ and $\sigma$ are measurable mappings from $\mathbb{R}_+\times[0,1]\times C(\mathbb{R}_+\times[0,1])$ into $\mathbb R$. Moreover, $K$ is a  measure which compensates reflections of the functional of $u$ of the form $\mathbb{E}\left[h(t,x,u(t,x))\right]\geq 0$ for a given continuous function $h$. In the spirit of Skorokhod condition for reflected equations, {\it the so-called flat solution}  (defined by Definition \ref{def solution} below) satisfies the extra condition
  \begin{align*}
  \int^T_0\int^1_0\mathbb{E}\left[h(t,x,u(t,x))\right] K(\mathrm{d}t,\mathrm{d}x)=0 \quad \text{for all}~T>0.
  \end{align*}
This condition intuitively indicates that $K$ is only allowed to push the random process $u$ when the constraint is binding.
These mean reflected SPDEs demonstrate considerable potential for application across a range of physical models. In complex systems governed by SPDEs, individual dynamics rarely act in isolation but are instead globally constrained by collective behavior. For example, in intracellular transport, the activation of membrane efflux channels requires that a key ion achieve a certain average concentration threshold; in ecological systems, preventing local species extinction is essential for maintaining overall ecosystem stability; in turbulent flows, regulating the average kinetic energy within a bounded range is crucial to suppressing excessive turbulence; and in electrochemical environments, the average current density at the electrode surface may also need to be maintained within a precise target interval.

 There have been a substantial number of investigations focusing on various aspects of reflection problems for SPDEs. Nualart and Pardoux  \cite{NP1992} first established  the well-posedness of the obstacle problem for the stochastic heat equation driven by the space-time white noise with a constant diffusion coefficient.  Then, the reflected SPDEs have been widely studied in recent years.  Research findings cover aspects such as the existence and uniqueness of solutions and asymptotic behavior, among other topics. Examples of relevant references include \cite{DMP1993,X2022,XZ2009,Zhang2010}, etc. None of the above-mentioned papers consider the case when the reflection is on the law of the solution.
In a notable extension, Duan, Hu and Peng  \cite{DuanJTP2025} considered  mean reflected SPDEs, where the compensating reflection part depends not on the paths but on the law of the solution. These constitute a new type of reflected SPDEs and enrich the reflected SPDE theory.

In this paper, we  focus on {\it the quadratic transportation cost inequality} for the solution of  Eq.\,\eqref{SPDE}.
Let us  now revisit  {\it the transportation cost inequality}.
 Suppose that $(E,d)$ is a metric space,  and let  $\mathcal{M}(E)$ denote the class of all probability measures on $E$. Given two probability measures $\mu, \nu\in \mathcal{M}(E)$ and $p\geq1$, {\it the Wasserstein distance} is defined by
\begin{align*}
W_p(\mu,\nu):=\inf_\pi\left[\int_E\int_Ed(x,y)^p\pi(\mathrm{d}x,\mathrm{d}y)\right]^{\frac{1}{p}},
\end{align*}
where the infimum is taken over all probability measures $\pi$ on  $E\times E$ with marginal distributions $\mu$ and $\nu$. The relative entropy of $\nu$ with respect to $\mu$ is defined by
\begin{align}\label{Def-H}
H(\nu|\mu):=
\begin{cases}
\displaystyle\int_E\log\frac{\mathrm{d}\nu}{\mathrm{d}\mu}\,\mathrm{d}\nu, & \text{if }\nu\ll\mu;\\[8pt]
+\infty, & \text{otherwise}.
\end{cases}
\end{align}
\begin{definition}{\rm
The probability measure $\mu$ is said to satisfy the transportation cost inequality $T_p(C)$ on $(E,d)$ if there exists a constant $C>0$ such that, for any probability measure $\nu$ on $\mathcal{M}(E)$,}
\begin{align}\label{Inequality-W_p}
W_p(\mu,\nu)\leq\sqrt{2CH(\nu|\mu)}.
\end{align}
\end{definition}
We write $\mu\in T_p(C)$ for this relation for short.
In fact, by H\"older's inequality,
\begin{align*}
T_{p'}(C)\subseteq T_{p}(C)\,\,\text{for}\,\,1\leq p \leq p' \,\,\text{and}\,\, C>0.
\end{align*}
Let us  recall some preliminaries regarding the two most particular  transportation cost inequalities, $T_1(C)$  and $T_2(C)$.  The inequality $T_1(C)$ is related to the concentration of measure phenomenon, see, e.g., \cite{Bobkov2000, Boucheron2013, DGW2004,Talagrand.1996}. If there exist constants $C$ and $r>0$ such that for every $\e>0$ and every Borel subset $A\subseteq  E$ with $\mu(A)\geq\frac{1}{2}$, $$\mu(A_{\e})\geq 1-Ce^{-r\e^2},$$
where $A_{\e}$ is the $\e$-neighborhood of $A$, i.e., $$A_{\e}=\{y:d(x,y)<\e\text{ for some } x\in A \},$$  then it is called that $\mu\in \mathcal{M}(E)$ has normal concentration on $(E,d)$.

The inequality $T_1(C)$ guarantees a normal concentration property for a probability measure (see, e.g., \cite{Marton1996, Ledoux2001, Boucheron2013,MartonK.1996}). To obtain such a concentration result, it is sufficient to prove the inequality \eqref{Inequality-W_p} for $p>1$  since  the $L^p$-transportation cost inequality \eqref{Inequality-W_p} becomes stronger as $p$  increases. The case $p=2$, known as  the quadratic transportation cost inequality  ($T_2(C)$), is of particular interest.  The  inequality  $T_2(C)$  stands out for two practical reasons: it allows us to use stochastic calculus and has the dimension free tensorization property. The  inequality  $T_2(C)$ was first established by Talagrand \cite{Talagrand19966} for the Gaussian measure. The approach of Talagrand was generalized by Feyel and \"Ust\"unel  \cite{FeyelU2004}
to the framework of the abstract Wiener space. Transportation cost inequalities also have close connections with  the logarithmic Sobolev and  Poincar\'e
inequalities, see, e.g., \cite{BobkovGotze1999,OV2000}.

 Numerous studies have investigated the $T_2(C)$ for S(P)DEs. A key starting point is the work of Djellout, Guillin and Wu \cite{DGW2004}, who studied the $T_2(C)$ for SDEs using Girsanov's transform  under the $L^2$
and the Cameron-Martin metrics. Their approach has inspired numerous subsequent studies, see, e.g., \cite{Guillin2014, Soumik2012, SSA2019}. For SDEs,
Bao et al. \cite{BaoWY2013} established the $T_2(C)$ for SDEs under both the uniform and the $L^2$ metrics. The framework was later extended to reflected SDEs by Boufoussi and Mouchtabih \cite{BM2019}, who also established the $T_2(C)$  under the uniform
and the $L^2$ metrics for laws of solutions of RSDEs. In the SPDE setting, early progress was made by Wu and Zhang \cite{WuZhang2006W}, who obtained $T_2(C)$   under the $L^2$ metric by  Galerkin approximation. Boufoussi and Hajji \cite{BH2018} later obtained $T_2(C)$ under
the $L^2$ metric for stochastic heat equations driven by the space-time white noise.
Khoshnevisan and Sarantsev \cite{KS2019} established the $T_2(C)$ for more general SPDEs under the $L^2$  and the uniform metrics in the case of additive noise. A significant breakthrough came from Shang and Zhang \cite{SZ2019}, who were the first to establish the  $T_2(C)$ under the uniform metric for stochastic heat equations driven by the multiplicative space-time white noise. Other related transportation cost inequalities in the SPDE context can be found in works such as \cite{LiLi2020,LiZhang2024,SSA2019, ShangWR2020,  WanfFengyuZTS2020}.

This paper aims to establish the  $T_2(C)$ for Eq.\,\eqref{SPDE} on  the space $C([0,T]\times[0,1])$.  In contrast to the  $T_2(C)$ for reflected SPDEs driven by space-time white noise under the uniform norm in \cite{LiZhang2024},  we consider the case where the   compensating reflection part depends not on the paths but on the law of the solution.

We mainly face two difficulties. The first difficulty stems from the fact that the compensating reflection
part depends not on the paths but on the law of the solution. To address this, we employ an estimate inspired by \cite{DuanJTP2025} showing that the difference in the uniform norm of solutions can be controlled by its expected value. The second difficulty comes from the multiplicative space-time white noise.  Due to technical restrictions, moment estimates for multiplicative space-time white noise hold only for large orders. To overcome this, we use a
$p$-th moment inequality that allows the order $p$ to be any positive number, in contrast to the requirement in \cite{SZ2019} that $p$ be sufficiently large.

The rest of this paper is organized as follows. In Section 2, we review the well-posedness of Eq.\,\eqref{SPDE} and state the main result of this paper. Section 3 and 4 are devoted to proving the main result. We provide some moment estimates for stochastic convolutions with respect to space-time white noise in the Appendix.

\section{Statement of the main result}
In this section, we first introduce the notation used throughout the paper and then state the main result.

We use the following notation in this paper. Let $G_t(x,y)$ be the fundamental solution of the heat equation with Dirichlet boundary condition. Let  $\langle \cdot, \cdot \rangle$ denote the scalar product in $L^2(0,1)$, and let $|\cdot|_{\infty}^t$ denote the uniform norm in $C([0,t]\times[0,1])$. The space $C^2_0([0,1])$ is the set of functions $\psi\in C^2([0,1])$ with $\psi(0)=\psi(1)=0$.

In order to establish the $T_2(C)$ of Eq.\,\eqref{SPDE}, we consider only the case where the process $K$ is restricted to be non-random. This is because if the reflection term is allowed to be random, it will lead to an infinite number of solutions to Eq\,\eqref{SPDE}.
For this, let us recall the definition of the deterministic flat solution of Eq.\,\eqref{SPDE} from \cite{DuanJTP2025}.
\begin{definition}{\rm(\cite[Definition 2.3]{DuanJTP2025})\label{def solution}
  A pair $(u,K)$  is called  a deterministic flat solution of the mean reflected SPDE  \eqref{SPDE} if the process $K$ is deterministic and}
\begin{itemize}
\item[(i)]  {\rm $u$ is a continuous random process on $[0,T]\times[0,1]$ with $\mathbb{E}\left[|u|_{\infty}^T\right]<\infty$ for any $T>0$; $u(t,x)$ is $\mathcal F_t$ measurable and satisfies the constraint \eqref{Restraincondition}.}
\item[(ii)] {\rm $K$ is a  measure on $\mathbb R_+\times(0,1)$ such that $K([0,T]\times(\e,1-\e))<\infty$ for every  small $\e>0$ and $T>0$.}
\item[(iii)]{\rm$(u,K)$ solves the parabolic SPDE in the following sense: for $\forall t\in \mathbb R_+$, $\varphi\in C^2_0([0,1])$,
\begin{align*}
&\langle u(t), \varphi \rangle-\int^t_0\langle u(s), \varphi'' \rangle \mathrm{d}s+\int_0^t\int_0^1   f(s,x, u) \varphi(x) \mathrm{d}s\mathrm{d}x\\
&\,=\langle u_0, \varphi \rangle+\int_0^t\int_0^1 \sigma(s,x, u)\varphi(x) W(\mathrm{d}s,\mathrm{d}x)+\int_0^t\int_0^1\varphi(x)K(\mathrm{d}s,\mathrm{d}x),\ \ \text{a.s.}.
\end{align*}}
\item[(iv)] {\rm $K$ increases only when needed. That is, $\int_0^{\infty}\int_0^1 \mathbb{E}\left[h(s,x,u(s,x))\right] K(\mathrm{d}s,\mathrm{d}x)=0$.}
\end{itemize}
\end{definition}

We next introduce the precise assumptions on the coefficients. $f$ and $\sigma$ are two measurable mappings,
\begin{align*}
f,\sigma: \mathbb{R}_+\times[0,1]\times C(\mathbb{R}_+\times[0,1])\rightarrow \mathbb{R},
\end{align*}
which satisfies
\begin{itemize}
\item[(I)]  {\rm for any $u, v\in  C(\mathbb{R}_+\times[0,1])$, $(t,x)\in\mathbb{R}_+\times[0,1]$ such that $u^t=v^t$,
\begin{align*}
f(t,x;u)=f(t,x;v)\quad\text{and}\quad
\sigma(t,x;u)=\sigma(t,x;v),
\end{align*}
where $u^t,v^t$ denote the restriction of $u,v$ to $[0,t]\times[0,1]$ respectively.}
\item[(II)] {\rm for any $T>0$, there exists constants $C_T$ and $M_T$ depending only on $T$ and the constant $M_{\sigma}$ such that for any $t\in[0,T]$, $x\in[0,1]$ and $u,v \in  C(\mathbb{R}_+\times[0,1])$,}
   \begin{align}\label{Lipschitz_fsigma}
  |f(t,x;u)-f(t,x;v)|+|\sigma(t,x;u)-\sigma(t,x;v)|\leq C_T|u-v|^t_{\infty},
  \end{align}
\begin{align}\label{bound_f}
  |f(t,x;u)|\leq M_T\left(1+|u|^t_{\infty}\right)
  \end{align}
 {\rm and}
  \begin{align}\label{bound_sigma}
  |\sigma(t,x;u)|\leq M_{\sigma}.
  \end{align}
\end{itemize}

In addition, we consider some assumptions on the constraint function $h$.

{\bf(H)} The continuous function $h: \mathbb{R}_+\times[0,1]\times\mathbb{R}\rightarrow \mathbb{R}$ is a measurable map satisfying
\begin{itemize}
\item[(i)] {\rm $\forall t\geq0$, $h(t,0,0)=h(t,1,0)=0$ and $y\mapsto h(t,x,y)$ is strictly increasing.}
\item[(ii)] {\rm $\forall x\in[0,1]$, $h(0,x,u_0(x))\geq 0$ and $h(t,x,0)\geq 0$, $h(t,x,\infty)\geq 0$ for any $(t,x)\in \mathbb{R}_+\times[0,1]$.}
 \item[(iii)] {\rm  $\forall t\in  \mathbb{R}_+$, $\forall x\in[0,1]$, $\forall y\in  \mathbb{R}$, $|h(t,x,y)|\leq C(1+|y|)$ for some positive constant $ C$. }
  \item[(iv)] {\rm $\forall t\in  \mathbb{R}_+$, $\forall x\in[0,1]$, $h$ is a bi-Lipschitz function in $y$: there exist positive constants $0<c_h\leq C_h$ such that}
\begin{align}\label{Lipschitz_h}
c_h|y-z|\leq|h(t,x,y)-h(t,x,z)|\leq C_h|y-z|,\qquad y,z\in\mathbb{R}.
\end{align}
\end{itemize}
%We assume the following precise conditions on the coefficients.
%\begin{condition}\label{con 1}
% Suppose that there are some positive constants $L_f$, $K_f$, $L_\sigma$ and $K_\sigma$ and $r\in \mathbb R$ such that, for some $\delta>0$,  $f$ and  $\sigma $ are  measurable mappings from  $\mathbb R_+\times\mathbb R$ to $ \mathbb R$ satisfying  the following conditions:
%\begin{itemize}
%\item[(I)] for every $x\in \mathbb{R}_+$, $u,v\in \mathbb R$,
%$$
%|f(x,u)-f(x,v)|\le L_f|u-v|;
%$$
%\item[(II)] for every $x\in \mathbb{R}_+$, $u \in \mathbb R$,
%$$
%|f(x,u) |\le K_f\left(e^{r x}+|u|\right);
%$$
%\item[(III)] for every $x\in \mathbb{R}_+$, $u,v\in \mathbb R$,
%$$
%|\sigma(x,u)-\sigma(x,v)|\le L_\sigma e^{-\delta x} |u-v|;
%$$
%\item[(IV)] for every $x\in\mathbb{R}_+$, $u\in \mathbb R$,
%$$
%|\sigma(x,u) |\le K_\sigma    e^{(r-\delta) x  }.
%$$
%\end{itemize}
%\end{condition}

According to     Duan, Hu and Peng \cite{DuanJTP2025}, we know the following result on the existence and uniqueness of the solution to Eq.\,\eqref{SPDE}.
\begin{proposition}{\rm (\cite[Theorem 3.3]{DuanJTP2025} and \cite[Theorem 4.4]{DuanJTP2025})\label{proposition2.1}}
\begin{itemize}
\item[(i)]{\rm
Under Assumptions (I) and (II), Eq.\,\eqref{SPDE} with the linear mean reflection admits a unique deterministic flat solution $u$.  Moreover, $\mathbb E\left[\left(|u|_{\infty}^T\right)^p\right]<\infty$ for any $p\geq1$. }
\item[(ii)]{\rm Under Assumptions (I), (II) and (H), Eq.\,\eqref{SPDE} with the general mean reflection admits a unique deterministic flat solution $u$.  Moreover, $\mathbb E\left[\left(|u|_{\infty}^T\right)^p\right]<\infty$ for any $p\geq1$. }
\end{itemize}
 \end{proposition}

Define
 \begin{align*}
C_{T,2,\e}=\inf_{q>10}\frac{2}{q-2}q^{-\frac{q}{2}}\e^{1-\frac{q}{2}}\left(q-2+qC_{T,q}\right)^{\frac{q}{2}}
\end{align*}
with
\begin{align*}
C_{T,q}<q^{\frac{q}{2}}T^{\frac{q}{4}-\frac{3}{2}}\left(\frac{2}{\pi}\right)^q
\left(\frac{1}{\sqrt{2\pi}}\right)^{\frac{q}{2}+1}\left(\frac{6q-8}{q-10}\right)^{\frac{3q}{2}-2},
\end{align*}
\begin{equation}\label{C1_Define}
\begin{split}
C_1:=&\, \inf\limits_{\substack{0<\e<\frac{1}{12C_T^2}}}
\Bigg\{\frac{12\sqrt{\frac{2T}{\pi}}M_\sigma^2}{1-12\e C_T^2}\cdot \exp{\Bigg(\frac{12C_T^2\left(\sqrt{\frac{2T}{\pi}}+C_{T,2,\e}\right)}{1-12\e C_T^2}\Bigg)}\Bigg\},
\end{split}
\end{equation}
\begin{equation}\label{C2_Define}
\begin{split}
C_2:=&\, \inf\limits_{\substack{0<\e<\frac{1}{6\left(1+C_h/c_h\right)C_T^2}}}
\Bigg\{\frac{6\left(1+C_h/c_h\right)\sqrt{\frac{2T}{\pi}}M_\sigma^2}{1-6\left(1+C_h/c_h\right)\e C_T^2}\\
&\qquad\qquad\qquad\quad\cdot \exp{\Bigg(\frac{6\left(1+C_h/c_h\right)C_T^2\left(\sqrt{\frac{2T}{\pi}}+C_{T,2,\e}\right)}{1-6\left(1+C_h/c_h\right)\e C_T^2}\Bigg)}\Bigg\}.
\end{split}
\end{equation}
We now present the main results of this paper.
\begin{theorem}\label{maintheorem1}
Suppose that Assumptions (I) and (II) hold.   Then the law $\mu$ of the solution to  Eq.\,\eqref{SPDE} with the linear mean reflection satisfies the quadratic transportation cost inequality $T_2(C_1)$ with the constant $C_1$ defined by \eqref{C1_Define} on the space $C\left([0,T]\times[0,1]\right)$.
\end{theorem}
\begin{theorem}\label{maintheorem2}
Suppose that Assumptions (I), (II) and (H) hold.  Then the law $\mu$ of the solution to  Eq.\,\eqref{SPDE} with the general mean reflection satisfies the quadratic transportation cost inequality $T_2(C_2)$ with the constant $C_2$ defined by \eqref{C2_Define} on the space $C\left([0,T]\times[0,1]\right)$.
\end{theorem}

\section{Proof of Theorem \ref{maintheorem1}}
In this section, we prove Theorem \ref{maintheorem1} for a special type of linear mean reflection, given by the function $h(t,x,Y)= Y-y(t,x)$. This means  the constraint takes the form
\begin{align}\label{Bound_ugeqy1}
\mathbb{E}\left[u(t,x)\right]\geq y(t,x),\qquad (t,x)\in\mathbb{R}_+\times[0,1],
\end{align}
where $y$ is a deterministic continuous map from $\mathbb{R}_+\times[0,1]$ to $\mathbb{R}$. Throughout this section, we assume $y(0,x)=0$ and $y(t,0)=y(t,1)=0$.

Before proving Theorem \ref{maintheorem1}, we first study the following deterministic parabolic obstacle problem
\begin{equation}\label{deterministicparabolicobstacleEq}
\left\{\begin{split}
   &\frac{\partial \bar{z}}{\partial t}(t, x)- \frac{\partial^2 \bar{z}}{ \partial x^2} (t,x)= \dot{ K}(t, x)\\
   &\bar{z}(t,x)\geq -v(t,x),\\
    &\int^T_0\int^1_0(\bar{z}+v)(t, x)K(\mathrm{d}t,\mathrm{d}x)=0,\qquad \text{for}\,\,\text{all}\,\, T>0.
 \end{split}\right.
\end{equation}
Here, the function $v$ belongs to $C([0,T]\times[0,1])$ and satisfies $v(0,x)=u_0(x)$.
Recall the following definition for the deterministic obstacle problem from Duan, Hu and Peng \cite{DuanJTP2025}.

\begin{definition}{\rm (\cite[Definition 3.1]{DuanJTP2025})\label{def3.1teach} A pair $(\bar{z},K)$ is called a solution to Eq.\,\eqref{deterministicparabolicobstacleEq} if it satisfies the following conditions}
\begin{itemize}
\item[(i)] {\rm $\bar{z}$ is continuous function on $[0,T]\times[0,1]$ satisfying $\bar{z}(t,x)\geq-v(t,x)$, $\bar{z}(t,0)=\bar{z}(t,1)=0  \text{ and } \bar{z}(0,x)=0$.}
\item[(ii)] {\rm $K$ \text{is a measure on} $(0,1)\times\mathbb R_+$ \text{such that}  $K\left((\e,1-\e)\times[0,T]\right)<\infty$ for every small $\e>0$ and $T>0$.}
\item[(iii)] {\rm $\bar{z}$ weakly solves the partial differential equation}
{\rm \begin{align}\label{eq DOP}
\frac{\partial \bar{z}}{\partial t}=\frac{\partial^2 \bar{z}}{\partial x^2}+\dot{K}.
\end{align}
\text{That is, for every} $t\in [0,T]$ and every $\varphi\in  C_0^2 ([0,1])$,
$$
\int_0^1 \bar{z}(t,x)\varphi(x)\mathrm{d}x=\int_0^t \int_0^1 \bar{z}(s,x)\varphi''(x)\mathrm{d}s\mathrm{d}x+\int_0^t \int_0^1 \varphi(x)K(\mathrm{d}s,\mathrm{d}x).
$$}
\item[(iv)] {\rm
$$
\int_0^t \int_0^1 \left(\bar{z}(s,x)+v(s,x)\right)K(\mathrm{d}s,\mathrm{d}x)=0
$$}
\end{itemize}
\end{definition}

In order to prove Theorem \ref{maintheorem1}, we need the following lemma describing all probability measures which are absolutely continuous with respect to $\mu$. Let $\nu\ll\mu$ on  $C([0,T]\times[0,1])$.  Define a new probability measure $\mathbb{Q}$ on the filtered probability space $(\Omega,\mathcal{F},\{\mathcal{F}\}_{0\leq t\leq T},\mathbb{P})$ by
\begin{align}\label{dQdP}
\mathrm{d}\mathbb{Q}:=\frac{\mathrm{d}\nu}{\mathrm{d}\mu}(u)\mathrm{d}\mathbb{P}.
\end{align}
Denote the Radon-Nikodym derivative restricted on $\mathcal{F}_t$ by
\begin{align*}
M_t:=\frac{\mathrm{d}\mathbb{Q}}{\mathrm{d}\mathbb{P}}\Big|_{\mathcal{F}_t},\,\,t\in[0,T].
\end{align*}
Then $\{M_t\}_{t\in[0,T]}$ is a nonnegative $\mathbb{P}$-martingale. The following lemma from \cite[Lemma 3.1]{KS2019} is crucial for our proof.
\begin{lemma}\label{numupq}{\rm(\cite[Lemma 3.1]{KS2019})}
{\rm There exists an adapted  random field $g=\left\{g(s,x)\right\}_{(s,x)\in[0,T]\times[0,1]}$ such that $\mathbb{Q}$-a.s. for all $t\in[0,T]$,}
\begin{align*}
\int^{t}_0\int_0^1 g^2(s,x)\mathrm{d}s\mathrm{d}x<\infty,
\end{align*}
{\rm and $\widetilde{W}:[0,T]\times[0,1]\rightarrow\RR$ defined by}
\begin{equation}\label{widetildeWbiaoshi}
\begin{split}
\widetilde{W}(t,x):=W(t,x)-\int^t_0\int_0^xg(s,y)\mathrm{d}s\mathrm{d}y,
\end{split}
\end{equation}
{\rm is a Brownian sheet under the measure $\mathbb{Q}$.
Moreover,}
\begin{align*}
M_t=\exp\left(\int^t_0\int^1_0g(s,x){W}(\mathrm{d}s,\mathrm{d}x)-\frac{1}{2}\int^{{t}}_0\int^1_0 g^2(s,x)\mathrm{d}s\mathrm{d}x\right),\,\,\mathbb{Q}\text{-a.s.},
\end{align*}
\begin{align}\label{Hnvmu}
H(\nu|\mu)=\frac{1}{2}\mathbb{E}^{\mathbb{Q}}\left[\int^T_0\int^1_0 g^2(s,x)\mathrm{d}s\mathrm{d}x\right],
\end{align}
{\rm where $\mathbb{E}^{\mathbb{Q}}$ stands for the expectation under the measure $\mathbb{Q}$.}
\end{lemma}

We now prove Theorem \ref{maintheorem1}.

\begin{proof}[Proof of Theorem \ref{maintheorem1}]
In view of the globally Lipschitz condition of the coefficients $f$ and $\sigma$, it follows from \cite{Walsh1986} that
\begin{align*}
z(t,x):=&\int^1_0G_t(x,y)u_0(y)\mathrm{d}y-\int^t_0\int^1_0G_{t-s}(x,y)f(s,y,u)\mathrm{d}s\mathrm{d}y\\
&+\int^t_0\int^1_0G_{t-s}(x,y)\sigma(s,y,u)W(\mathrm{d}s,\mathrm{d}y),
\end{align*}
satisfies the SPDE
\begin{equation}\label{SPDEz}
\left\{\begin{split}
   &\frac{\partial z}{\partial t}(t, x)- \frac{\partial^2 z}{ \partial x^2} (t,x) + f(t,x,u(t,x))=\sigma(t,x,u(t,x)) \dot W(t, x) \\
 % &\qquad\qquad\quad +\dot W(t,x), \quad t\ge 0, x\in [0,1],\\
   &z(0,x)=u_0(x),\\
    &z(t,0)=z(t,1)=0.
 \end{split}\right.
\end{equation}
Let $(\bar{z},K)$ be the unique deterministic solution of the obstacle problem \eqref{deterministicparabolicobstacleEq} with $v(t,x)=\mathbb{E}[z(t,x)]-y(t,x)$. Set $u=z+\bar{z}$, then it is easy to verify that $(u,K)$ is the unique solution of Eq.\,\eqref{SPDE}.

Let $\mathbb{Q}$ be defined as \eqref{dQdP} and let $g$ be the corresponding process appeared in Lemma \ref{numupq}.
Then the solution $u(t,x)$ of Eq.\,\eqref{SPDE} satisfies the following SPDE under the measure $\mathbb{Q}$,
\begin{equation}\label{1}
\left\{\begin{split}
   &\frac{\partial u}{\partial t}(t, x)- \frac{\partial^2 u}{ \partial x^2} (t,x) + f(t,x,u(t,x))=\sigma(t,x,u(t,x)) {\widetilde{W}}(\mathrm{d}t, \mathrm{d}x) \\
   &\qquad\qquad\quad+\sigma(t,x,u(t,x))g(t,x)+K(\mathrm{d}t,\mathrm{d}x),\quad x\in [0,1],t\geq0\\
   &u(0,x)=u_0(x),\quad x\in [0,1],\\
    &u(t,0)=u(t,1)=0,\quad t\ge 0.
 \end{split}\right.
\end{equation}
By Lemma \ref{numupq}, we have $\mathbb{Q}$-a.s.
$\int^{t}_0\int_0^1 g^2(s,x)\mathrm{d}s\mathrm{d}x<\infty$ for all $t\in [0,T]$.
Therefore, by the proof of  \cite[Theorem 5.6]{DuanJTP2025}, we have that, for any $T>0$,
\begin{align}\label{boundu_u}
\mathbb{E}^{\mathbb{Q}}\left[\left(\left| u\right|_{\infty}^T\right)^2\right]<\infty.
\end{align}
Consider the solution to the following SPDE:
\begin{equation}\label{2}
\left\{\begin{split}
   &\frac{\partial \widetilde{u}}{\partial t}(t, x)- \frac{\partial^2 \widetilde{u}}{ \partial x^2} (t,x) + f(t,x,\widetilde{u}(t,x))=\sigma(t,x,\widetilde{u}(t,x)) {\widetilde{W}}(\mathrm{d}t, \mathrm{d}x) \\ &\qquad\qquad\quad+\sigma(t,x,u(t,x))g(t,x)+K(\mathrm{d}t,\mathrm{d}x),\quad x\in [0,1],t\geq0\\
   &\widetilde{u}(0,x)=u_0(x),\quad x\in [0,1],\\
    &\widetilde{u}(t,0)=\widetilde{u}(t,1)=0,\quad t\ge 0.
 \end{split}\right.
\end{equation}
It follows from Lemma \ref{numupq} that the law of $(\widetilde{u}, u)$ forms a coupling of $(\mu,\nu)$ under the measure $\mathbb{Q}$. Therefore, by the definition of the Wasserstein distance, we have
\begin{align}\label{4.8label}
W_2(\nu,\mu)^2\leq \mathbb{E}^{\mathbb{Q}}\left[\left(\left| u-\widetilde{u}\right|_{\infty}^T\right)^2\right].
\end{align}
In light of \eqref{Hnvmu} and \eqref{4.8label}, Theorem \ref{maintheorem1} will follow, once we  prove that
\begin{align}\label{suffitioncondition}
\mathbb{E}^{\mathbb{Q}}\left[\left(\left| u-\widetilde{u}\right|_{\infty}^T\right)^2\right]
\leq C_1\mathbb{E}^{\mathbb{Q}}\left[ \int^T_0\int^1_0g^2(s,y)\mathrm{d}s\mathrm{d}y\right],
\end{align}
where $C_1$ is a  constant independent of $g$.

The solution $z(t,x)$ of Eq.\,\eqref{SPDEz} satisfies the following SPDE under the measure $\mathbb{Q}$,
\begin{equation}\label{eq v}
\begin{split}
z(t,x):=&\int^1_0G_t(x,y)u_0(y)\mathrm{d}y-\int^t_0\int^1_0G_{t-s}(x,y)f(s,y,u)\mathrm{d}s\mathrm{d}y\\
&+\int^t_0\int^1_0G_{t-s}(x,y)\sigma(s,y,u)\widetilde{W}(\mathrm{d}s,\mathrm{d}y),\\
&+\int_0^t\int_0^1G(t-s, x,y)\sigma(s,y,u)g(s,y)\mathrm{d}s\mathrm{d}y.
\end{split}
\end{equation}
Define
\begin{equation}\label{eq barv}
\begin{split}
\widetilde{z}(t,x):=&\int^1_0G_t(x,y)u_0(y)\mathrm{d}y-\int^t_0\int^1_0G_{t-s}(x,y)f(s,y,\widetilde{u})\mathrm{d}s\mathrm{d}y\\
&+\int^t_0\int^1_0G_{t-s}(x,y)\sigma(s,y,\widetilde{u})\widetilde{W}(\mathrm{d}s,\mathrm{d}y).
\end{split}
\end{equation}
Then $\widetilde{\bar{z}}:=\widetilde{u}-\widetilde{z}$ and the pair $(\widetilde{z}, \widetilde{K})$ is the solution of Eq.\,\eqref{2}. By a similar argument as in \cite[Theorem 1.4]{NP1992}, \eqref{1}, \eqref{2}, \eqref{eq v}, \eqref{eq barv} and Jensen's inequality, we have
\begin{align*}
|\bar{z}-\widetilde{\bar{z}}|^T_{\infty}\leq&\sup_{t\in[0,T]}\sup_{x\in[0,1]}\left|\left(\mathbb{E}[z(t,x)]-y(t,x)\right)-\left(\mathbb{E}[\widetilde{z}(t,x)]-y(t,x)\right)\right|\\
\leq&\sup_{t\in[0,T]}\sup_{x\in[0,1]}\mathbb{E}\left[|z(t,x)-\widetilde{z}(t,x)|\right]\\
\leq&\mathbb{E}\left[|z-\widetilde{z}|^T_{\infty}\right],
\end{align*}
which implies
\begin{equation}\label{bound utitleu}
\begin{split}
|u-\widetilde{u}|^T_{\infty}\leq&|z-\widetilde{z}|^T_{\infty}+|\bar{z}-\widetilde{\bar{z}}|^T_{\infty}\\
\leq&|z-\widetilde{z}|^T_{\infty}+\mathbb{E}\left[|\bar{z}-\widetilde{\bar{z}}|^T_{\infty}\right],
\end{split}
\end{equation}
almost surely.
Then putting \eqref{1}, \eqref{2}, \eqref{eq v}, \eqref{eq barv} and \eqref{bound utitleu} together, we have that by Jensen's inequality,
\begin{equation}\label{4.14uubar}
\begin{split}
\mathbb{E}^{\mathbb{Q}}\left[\left(|u-\widetilde{u}|^T_{\infty}\right)^2\right]
\leq &\mathbb{E}^{\mathbb{Q}}\left[\Big(|z-\widetilde{z}|^T_{\infty}+\mathbb{E}\left[|\bar{z}-\widetilde{\bar{z}}|^T_{\infty}\right]\Big)^2\right]\\
\leq&2\mathbb{E}^{\mathbb{Q}}\left[\Big(|z-\widetilde{z}|^T_{\infty}\Big)^2+\Big(\mathbb{E}\left[|\bar{z}-\widetilde{\bar{z}}|^T_{\infty}\right]\Big)^2\right]\\
=&4\mathbb{E}^{\mathbb{Q}}\left[\Big(|z-\widetilde{z}|^T_{\infty}\Big)^2\right]\\
\leq&12(I_1+I_2+I_3),
\end{split}
\end{equation}
where
\begin{align*}
I_1:=&\,\mathbb{E}^{\mathbb{Q}}\left[\sup_{t\in[0,T]}\sup_{x\in[0,1]}\left|\int_0^t\int_0^1G_{t-s}( x,y)\left[f(s,y,u)-f(s,y,\widetilde{u})\right] \mathrm{d}s\mathrm{d}y\right|^2\right];\\
I_2:=&\,\mathbb{E}^{\mathbb{Q}}\left[\sup_{t\in[0,T]}\sup_{x\in[0,1]}\left|\int_0^t\int_0^1G_{t-s}( x,y)\left[\sigma(s,y,u)-\sigma(s,y,\widetilde{u})\right] \widetilde{W}(\mathrm{d}s,\mathrm{d}y)\right|^2\right];\\
I_3:=&\,\mathbb{E}^{\mathbb{Q}}\left[\sup_{t\in[0,T]}\sup_{x\in[0,1]}\left|\int_0^t\int_0^1G_{t-s}( x,y)\sigma(s,y,u)g(s,y) \mathrm{d}s\mathrm{d}y\right|^2\right].
\end{align*}
By \eqref{Lipschitz_fsigma}, \eqref{Bound_G2integral} and H\"older's inequality, we have
\begin{equation}\label{I1uubar}
\begin{split}
I_1\leq&\,C_T^2\left(\sup_{t\in[0,T]}\sup_{x\in[0,1]}\int_0^t \int_0^1G^2_{t-s}( x,y)\mathrm{d}s\mathrm{d}y\right)\cdot\mathbb{E}^{\mathbb{Q}}\left[\int^T_0\left(|u-\widetilde{u}|^s_{\infty}\right)^2ds\right]\\
\leq&\,\sqrt{\frac{2T}{\pi}}C_T^2\mathbb{E}^{\mathbb{Q}}\left[\int^T_0\left(|u-\widetilde{u}|^s_{\infty}\right)^2\mathrm{d}s\right]\\
\leq&\,\sqrt{\frac{2T}{\pi}}C_T^2\int^T_0\mathbb{E}^{\mathbb{Q}}\left[\left(|u-\widetilde{u}|^s_{\infty}\right)^2\right]\mathrm{d}s,
\end{split}
\end{equation}
where the constant $C_T$ is  defined in \eqref{Lipschitz_fsigma}. By \eqref{Lipschitz_fsigma} and Lemma \ref{lem int34}  below, we have that
\begin{equation}\label{I2uubar}
\begin{split}
I_2\leq&\,\e\,\mathbb{E}^{\mathbb{Q}}\left[\sup_{(t,x)\in[0,T]\times[0,1]}\left|\sigma(t,x, u)-\sigma(t,x, \widetilde{u})\right|^2\right]\\
&+C_{T,2,\e}\,\mathbb{E}^{\mathbb{Q}}\left[\int^T_0\sup_{y\in[0,1]}\left|\sigma(s,y, u)-\sigma(s,y, \widetilde{u})\right|^2\mathrm{d}s\right]\\
\leq&\,\e C_T^2\,\mathbb{E}^{\mathbb{Q}}\left[\sup_{(t,x)\in[0,T]\times[0,1]}\left|  u(t,x)- \widetilde{u}(t,x)\right|^2\right]\\
&+C_T^2C_{T,2,\e}\,\int^T_0\mathbb{E}^{\mathbb{Q}}\left[\sup_{(r,y)\in[0,s]\times[0,1]}\left| u(r,y)- \widetilde{u}(r,y)\right|^2\right]\mathrm{d}s,
\end{split}
\end{equation}
where  the constant $C_{T,2,\e}$ is   defined in \eqref{DefineCTpe} with $p$ replaced by $2$. By \eqref{bound_sigma}, \eqref{Bound_G2integral} and H\"older's inequality, we have
\begin{equation}\label{I3uubar}
\begin{split}
I_3\leq&\,\mathbb{E}^{\mathbb{Q}}\left[\sup_{t\in[0,T]}\sup_{x\in[0,1]}\left|\int_0^t\int_0^1G^2_{t-s}( x,y)\sigma^2(s,y,u) dsdy\right|\cdot\int^T_0\int^1_0g^2(s,y)\mathrm{d}s\mathrm{d}y\right]\\
\leq&\,M_\sigma^2\mathbb{E}^{\mathbb{Q}}\left[\sup_{t\in[0,T]}\sup_{x\in[0,1]}\left|\int_0^t\int_0^1G^2_{t-s}( x,y) dsdy\right|\cdot\int^T_0\int^1_0g^2(s,y)\mathrm{d}s\mathrm{d}y\right]\\
\leq&\,\sqrt{\frac{2T}{\pi}}M_\sigma^2\mathbb{E}^{\mathbb{Q}}\left[\int^T_0\int^1_0g^2(s,y)\mathrm{d}s\mathrm{d}y\right],
\end{split}
\end{equation}
where the constant $M_\sigma$ is  defined in \eqref{bound_sigma}.
Putting \eqref{4.14uubar}, \eqref{I1uubar}, \eqref{I2uubar} and \eqref{I3uubar} together, we have
\begin{align*}
\mathbb{E}^{\mathbb{Q}}\left[\left(|u-\widetilde{u}|^T_{\infty}\right)^2\right]
\leq &12\e C_T^2\,\mathbb{E}^{\mathbb{Q}}\left[\left(|u-\widetilde{u}|^T_{\infty}\right)^2\right]\\
&+12C_T^2\left(\sqrt{\frac{2T}{\pi}}+C_{T,2,\e}\right)\int^T_0\mathbb{E}^{\mathbb{Q}}\left[\left(|u-\widetilde{u}|^s_{\infty}\right)^2\right]\mathrm{d}s\\
&+12\sqrt{\frac{2T}{\pi}}M_\sigma^2\mathbb{E}^{\mathbb{Q}}\left[\int^T_0\int^1_0g^2(s,y)\mathrm{d}s\mathrm{d}y\right].
\end{align*}
Taking any $0<\e<\frac{1}{12C_T^2}$,
\begin{align*}
\mathbb{E}^{\mathbb{Q}}\left[\left(|u-\widetilde{u}|^T_{\infty}\right)^2\right]
\leq &\frac{12C_T^2\left(\sqrt{\frac{2T}{\pi}}+C_{T,2,\e}\right)}{1-12\e C_T^2}\int^T_0\mathbb{E}^{\mathbb{Q}}\left[\left(|u-\widetilde{u}|^s_{\infty}\right)^2\right]\mathrm{d}s\\
&+\frac{12\sqrt{\frac{2T}{\pi}}M_\sigma^2}{1-12\e C_T^2}\left[\int^T_0\int^1_0g^2(s,y)\mathrm{d}s\mathrm{d}y\right].
\end{align*}
By Proposition \ref{proposition2.1}, we have that, for any $T>0$,
\begin{align}\label{bounduinfty}
\mathbb{E}^{\mathbb{Q}}\left[\left(|\widetilde{u}|^T_{\infty}\right)^2\right]<\infty,
\end{align}which together with \eqref{boundu_u} yields
$\mathbb{E}^{\mathbb{Q}}\left[\left(|u-\widetilde{u}|^T_{\infty}\right)^2\right]<\infty$ for any $T>0$.
Therefore, Gr\"onwall's inequality implies that \eqref{suffitioncondition} holds with
\begin{align*}
C_1=&\, \inf\limits_{\substack{0<\e<\frac{1}{12C_T^2}}}
\Bigg\{\frac{12\sqrt{\frac{2T}{\pi}}M_\sigma^2}{1-12\e C_T^2}\cdot \exp{\Bigg(\frac{12C_T^2\left(\sqrt{\frac{2T}{\pi}}+C_{T,2,\e}\right)}{1-12\e C_T^2}\Bigg)}\Bigg\}.
\end{align*}
The proof is complete.
\end{proof}
\section{Proof of Theorem \ref{maintheorem2}}
In this section, we prove Theorem \ref{maintheorem2} under the constraint \eqref{Restraincondition}. This proof deals with the general case of mean reflection. For this, we first consider the following parabolic obstacle problem
\begin{equation}\label{deterEqGeneral1}
\left\{\begin{split}
   &\frac{\partial \bar{z}}{\partial t}(t, x)- \frac{\partial^2 \bar{z}}{ \partial x^2} (t,x)= \dot{ K}(t, x),\\
   &\mathbb{E}\left[h(t,x,(\bar{z}+v)(t,x))\right]\geq0,\\
    &\int^T_0\int^1_0\mathbb{E}\left[h(t,x,(\bar{z}+v)(t,x))\right]\mathrm{d}K(t,x)=0,\qquad \text{for}\,\,\text{all}\,\, T>0,
 \end{split}\right.
\end{equation}
where $v$ is a continuous $\mathcal{F}_t$-adapted process taking values in  $C([0,T]\times[0,1])$ and $\mathbb{E}\left[|v|^T_{\infty}\right]<\infty$ for any $T>0$. Moreover, $v(0,x)\geq 0$ and $v(t,0)=v(t,1)=0$ for all $x\in[0,1]$, $t\geq0$. It is noted that $\bar{z}$ should be taken as the first component of the solution to the obstacle problem\eqref{deterEqGeneral1} with $v=z$, i.e.,
\begin{align*}
z(t,x):=&\int^1_0G_t(x,y)u_0(y)\mathrm{d}y-\int^t_0\int^1_0G_{t-s}(x,y)f(s,y,u)\mathrm{d}s\mathrm{d}y\\
&+\int^t_0\int^1_0G_{t-s}(x,y)\sigma(s,y,u)W(\mathrm{d}s,\mathrm{d}y),
\end{align*}
and $(\bar{z},K)$ is the unique deterministic solution of the obstacle problem  \eqref{deterEqGeneral1} with  $v=z$. Set $u:=z+\bar{z}$, then $(u,K)$ is the unique deterministic flat solution of Eq.\,\eqref{SPDE} with the constraint \eqref{Restraincondition} (see \cite{DuanJTP2025}).

The following  Lemma \ref{Lemmestz1z2} establishes that, under Assumption (H), the differences in the uniform norm of the solution to the obstacle  problem  \eqref{deterEqGeneral1} can be controlled by the expectation of differences in the uniform norm of obstacles. This result plays a crucial role in SPDEs with general mean reflection. Let $\mathcal{S}^1$ be the set of continuous $\mathcal{F}_t$-adapted process on $C([0,T]\times[0,1])$ with $\mathbb{E}\left[|v|^T_{\infty}\right]<\infty$ for any $T>0$.

\begin{lemma}{\rm(\cite[Lemma 4.3]{DuanJTP2025})\label{Lemmestz1z2}}
{\rm Let  $\bar{z}_i$ be the first component of the solution to  the obstacle problem  \eqref{deterEqGeneral1} corresponding to process $v_i\in \mathcal{S}^1$, $i=1,2$, respectively. Then under the Assumption (H), we have for any $T>0$,}
\begin{align}\label{estz1z2}
|\bar{z}_1-\bar{z}_2|^T_{\infty}\leq l\mathbb{E}\left[|v_1-v_2|^T_{\infty}\right],
\end{align}
{\rm where $l\geq 0$ is a constant.}
\end{lemma}
\begin{remark}\label{Remark1}
{\rm It follows from \cite[Lemma 8]{BEH2018} and the proof of \cite[Lemma 4.1]{DuanJTP2025} that $l=\frac{C_h}{c_h}$  {\rm(}$0<c_h\leq C_h$ {\rm)} in \eqref{estz1z2}, where  $c_h$ and  $C_h$ are  positive constants defined  in  Assumption (H).}
\end{remark}

We now  prove  Theorem \ref{maintheorem2}.
\begin{proof}[The proof of Theorem \ref{maintheorem2}]
Similar to \eqref{bound utitleu}, by Lemma \ref{Lemmestz1z2} and Remark \ref{Remark1}, we have
\begin{equation}\label{bound utitleu2}
\begin{split}
|u-\widetilde{u}|^T_{\infty}\leq&|z-\widetilde{z}|^T_{\infty}+|\bar{z}-\widetilde{\bar{z}}|^T_{\infty}\\
\leq&|z-\widetilde{z}|^T_{\infty}+\frac{C_h}{c_h}\mathbb{E}\left[|\bar{z}-\widetilde{\bar{z}}|^T_{\infty}\right],
\end{split}
\end{equation}
almost surely,  where  $c_h$ and  $C_h$ are  positive constants defined  in  Assumption (H).  We now reproduce the proof of Theorem \ref{maintheorem1}, following the same argument as in  \eqref{suffitioncondition}, with the constant $C_1$ replaced by
\begin{align*}
C_2=&\, \inf\limits_{\substack{0<\e<\frac{1}{6\left(1+C_h/c_h\right)C_T^2}}}
\Bigg\{\frac{6\left(1+C_h/c_h\right)\sqrt{\frac{2T}{\pi}}M_\sigma^2}{1-6\left(1+C_h/c_h\right)\e C_T^2}\\
&\qquad\qquad\qquad\quad\cdot \exp{\Bigg(\frac{6\left(1+C_h/c_h\right)C_T^2\left(\sqrt{\frac{2T}{\pi}}+C_{T,2,\e}\right)}{1-6\left(1+C_h/c_h\right)\e C_T^2}\Bigg)}\Bigg\}.
\end{align*}
The proof is complete.
\end{proof}

\section{Appendix}
In this section, let us establish some estimates,  which play a crucial role in the proof of the main result.

Recall that the function $G$ is the Dirichlet heat kernel on $[0,1]$, i.e., for any $t>0, x, y\in[0,1]$,
\begin{equation}\label{eq kernel}
G_t(x,y):=\frac{1}{\sqrt{4\pi t}}\sum_{n=-\infty}^{\infty}\left[\exp\left(-\frac{(x-y+2n)^2}{4t} \right)-\exp\left(-\frac{(x+y+2n)^2}{4t} \right)\right],
\end{equation}
and the well-known Nash-Aronson estimate (see, e.g., \cite{Dalang2009})
\begin{align*}
0\leq G_t(x,y)\leq\frac{1}{\sqrt{2\pi t}}\exp\left(-\frac{(x-y)^2}{2t} \right), \forall x,y\in[0,1].
\end{align*}
It follows from (3.9) in \cite{SZ2019} that
\begin{align*}
\int^1_0G^2_t(x,y)\mathrm{d}y\leq \sup_{y\in[0,1]} G_t(x,y)\cdot\int^1_0G_t(x,y)\mathrm{d}y\leq \frac{1}{\sqrt{2\pi t}}.
\end{align*}
That is, \begin{align}\label{Bound_G2integral}
\int^T_0\int^1_0G^2_t(x,y)\mathrm{d}s\mathrm{d}y \leq \frac{2T}{\sqrt{\pi }}.
\end{align}

\begin{lemma}{\rm (\cite[Proposition 3.4]{SZ2019})\label{lem int34}}
{\rm Let $\left\{\sigma(s,y):(s,y)\in\mathbb{R}_+\times[0,1]\right\}$ be a random field such that the stochastic integral against space time white noise is well defined. Then for any $T>0$, $0<p\leq 10$ and $\e>0$, there exists a constant $C_{T,p,\e}$ such that}
\begin{equation}\label{Est327}
\begin{split}
&\mathbb{E}\left[\sup_{(t,x)\in[0,T]\times[0,1]}\left|\int^t_0\int^1_0G_{t-s}(x,y)\sigma(s,y)W(\mathrm{d}s,\mathrm{d}y)\right|^p\right]\\
\leq&\e\mathbb{E}\left[\sup_{(t,x)\in[0,T]\times[0,1]}\left|\sigma(s,y)\right|^p\right]
+C_{T,p,\e}\mathbb{E}\int^T_0\sup_{y\in[0,1]}\left|\sigma(s,y)\right|^p\mathrm{d}s,
\end{split}
\end{equation}
{\rm where}

\begin{align}\label{DefineCTpe}
C_{T,p,\e}=\inf_{q>10}\frac{p}{q-p}q^{-\frac{q}{p}}\e{1-\frac{q}{p}}\left(q-p+qC_{T,q}\right)^{\frac{q}{p}}
\end{align}
{\rm with}
\begin{align*}
C_{T,q}<q^{\frac{q}{2}}T^{\frac{q}{4}-\frac{3}{2}}\left(\frac{2}{\pi}\right)^q
\left(\frac{1}{\sqrt{2\pi}}\right)^{\frac{q}{2}+1}\left(\frac{6q-8}{q-10}\right)^{\frac{3q}{2}-2}.
\end{align*}
\end{lemma}

\vskip0.5cm
\noindent{\bf Acknowledgments}: %The authors are grateful to the anonymous referees for comments.
%The completion of this article is inseparable from the guidance of Ran Wang.
%The authors are  grateful to Ran Wang for his good
%suggestions and timely help.
The research of B. Zhang is partially supported by the NSF for Youths of Jiangsu Province (No. BK20251077). The research of B. Qian is partially supported by the NSF of China (No. 11671076).

\medskip

\bigskip

\end{document}